\documentclass[a4paper,10pt]{scrartcl}
\usepackage{amsmath,amsthm,amssymb}
\usepackage[utf8]{inputenc}
\usepackage[T1]{fontenc}
\usepackage[a4paper]{geometry}
\usepackage{xcolor}

\usepackage{newunicodechar}
\newunicodechar{ℝ}{\mathbb{R}}
\newunicodechar{ℕ}{\mathbb{N}}
\newunicodechar{η}{\eta}
\newunicodechar{γ}{\gamma}
\newunicodechar{α}{\alpha}
\newunicodechar{β}{\beta}
\newunicodechar{ε}{\varepsilon}
\newunicodechar{ξ}{\xi}
\newunicodechar{∞}{\infty}
\newunicodechar{Δ}{\Delta}
\newunicodechar{σ}{\sigma}
\newunicodechar{δ}{\delta}
\newunicodechar{ζ}{\zeta}
\newunicodechar{χ}{\chi}
\newunicodechar{π}{\pi}
\newunicodechar{μ}{\mu}
\newunicodechar{τ}{\tau}
\newunicodechar{ω}{\omega}
\newunicodechar{ρ}{\rho}

\newcommand{\calS}{\mathcal{S}}
\newcommand{\Om}{\Omega}
\newcommand{\sub}{\subset}
\newcommand{\utilde}{\widetilde{u}}

\newcommand{\calF}{\mathcal{F}}
\newcommand{\calD}{\mathcal{D}}
\newcommand{\f}[2]{\frac{#1}{#2}}
\newcommand{\kl}[1]{\left(#1\right)}
\newcommand{\na}{\nabla}
\newcommand{\io}{\int_{\Omega}}
\newcommand{\bdry}{\big|_{\partial \Omega}}
\newcommand{\delny}{\partial_{\nu}}
\newcommand{\Lom}[1]{L^{#1}(\Omega)}
\newcommand{\omn}{\omega_{N}}
\newcommand{\norm}[2][]{\|#2\|_{#1}}
\newcommand{\normm}[2]{\|#2\|_{#1}}
\newcommand{\nn}{\nonumber}
\newcommand{\set}[1]{\{#1\}}
\newcommand{\Stilde}{\widetilde{S}}

\newcommand{\Ombar}{\overline{\Omega}}
\newcommand{\Tmax}{T_{max}}
\newcommand{\uehat}{\widehat{u}_{η}}
\newcommand{\vehat}{\widehat{v}_{η}}
\newcommand{\Dbar}{\overline{D}}
\newcommand{\uhat}{\widehat{u}}
\newcommand{\vhat}{\widehat{v}}

\newcommand{\ue}{u_{η}}
\newcommand{\ve}{{v_{η}}}
\newcommand{\re}{r_{η}}
\newcommand{\gh}{\frac{γ}2}

\newtheorem{theorem}{Theorem}[section]
\newtheorem{lemma}[theorem]{Lemma}
\newtheorem{corollary}[theorem]{Corollary}
\theoremstyle{definition}

\newtheorem{remark}[theorem]{Remark}
\allowdisplaybreaks

\title{Infinite time blow-up of many solutions to a general quasilinear parabolic-elliptic Keller--Segel system}
\author{Johannes Lankeit\thanks{Institut f\"ur Mathematik, Universit\"at Paderborn, Warburger Str.100, 33098 Paderborn, Germany; email: johannes.lankeit@math.uni-paderborn.de}}
\date{}

\begin{document}
\maketitle
\begin{abstract}
\noindent We consider a parabolic-elliptic chemotaxis system generalizing 
\begin{align*}
  u_t&=\nabla\cdot((u+1)^{m-1}\nabla u)-\nabla \cdot(u(u+1)^{\sigma-1}\nabla v)\\
  0&= \Delta v - v + u
\end{align*}
in bounded smooth domains $\Om\subset ℝ^N$, $N\ge 3$, and with homogeneous Neumann boundary conditions. We show that 
\begin{itemize}
\item solutions are global and bounded if $σ<m-\f{N-2}N$
\item solutions are global if $σ\le 0$
\item close to given radially symmetric functions there are many initial data producing unbounded solutions if $σ>m-\f{N-2}N$. 
\end{itemize}
In particular, if $σ\le 0$ and $σ>m-\f{N-2}N$, there are many initial data evolving into solutions that blow up after infinite time.\\

\noindent \textbf{Keywords:} Keller-Segel; chemotaxis; infinite time blow-up; unboundedness; global existence\\
\noindent \textbf{Math Subject Classification (2010):} 35B44, 92C17, 35Q92, 35A01, 35K55
\end{abstract}
%35B44 blowup
% 35A01
%35B65 View Publications (1980-now) Smoothness and regularity of solutions
%35B40 View Publications (1973-now) Asymptotic behavior of solutions
% 35K55 View Publications (1973-now) Nonlinear parabolic equations
% 35Q30 View Publications (1991-Now) Navier-Stokes equations [See also 76D05, 76D07, 76N10]
% 76D05 View Publications (1980-now) Navier-Stokes equations [See also 35Q30]
% 76B03 View Publications (2000-now) Incompressible, inviscid fluids: Existence, uniqueness, and regularity theory [See also 35Q35]
% 35Q35 View Publications (1991-Now) PDEs in connection with fluid mechanics
% 92C17 View Publications (2000-now) Cell movement (chemotaxis, etc.)
% 35B35 View Publications (1973-now) Stability
%  35A01 View Publications (2010-now) Existence problems: global existence, local existence, non-existence 
%  35D View Publications (1973-now) Generalized solutions
% 35D05 View Publications (1973-2009) Existence of generalized solutions
% 35D10 View Publications (1973-2009) Regularity of generalized solutions
% 35D30 View Publications (2010-now) Weak solutions
% 35D35 View Publications (2010-now) Strong solutions
% 35D40 View Publications (2010-now) Viscosity solutions
% 35D99 View Publications (1973-now) None of the above, but in this section 
% 35K65 View Publications (1980-now) Degenerate parabolic equations 
 
\section{Introduction}
Whereas diffusion has an equilibrating effect, cross-diffusive terms appearing in chemotaxis models like
\begin{align}\label{eq:firstsystem}
 u_t&=\nabla\cdot((u+1)^{m-1}\nabla u)-\nabla \cdot(u(u+1)^{\sigma-1}\nabla v)\\
  0&= \Delta v - v + u\nn
\end{align}
tend to lead to the exact opposite, to aggregation.
It is therefore of interest to characterize which of these mechanisms is more decisive for the solution behaviour, in dependence on their relative strengths as given by the size of the exponents $m$ and $\sigma$ in \eqref{eq:firstsystem}. Are all solutions global and bounded? Do some solutions blow up? If so, in finite or in infinite time? 
Indeed, there are studies showing for the related parabolic-parabolic chemotaxis system (where $0$ in the second equation of \eqref{eq:firstsystem} is replaced by $v_t$) that for different choices of $m$ and $\sigma$, any of these qualitatively different behaviours can be observed. One of the main tools for proving the existence of unbounded solutions is the use of an energy functional together with the construction of suitable initial data $u_0$, $v_0$ -- which makes this one of the few respects in which the fully parabolic system is easier to deal with than the parabolic-elliptic ``simplification'' \eqref{eq:firstsystem}. After all, there it is possible to choose $u_0$ and $v_0$ independently of each other, whereas in \eqref{eq:firstsystem} only $u_0$ can be selected, providing us with much less freedom for the construction. 
It is the parabolic--elliptic setting we are going to consider here; mainly being interested in $\sigma\le 0$.\\ 

Before we do so, let us briefly recall some of the known results in related models: 

We will begin with the \textbf{parabolic-parabolic model} 
\begin{align}\label{eq:pp}
u_t&=\nabla\cdot(D(u)\nabla u)-\nabla \cdot(S(u)\nabla v)\\
  v_t&= \Delta v - v + u \nn
\end{align}
in bounded domains $\Om\sub ℝ^N$ with no-flux boundary conditions and for easier comparison state the results for $D(u)=(u+1)^{m-1}$ and $S(u)=u(u+1)^{\sigma-1}$ or $S(u)=u^\sigma$ and let \[\alpha=m-σ-1,\] so that $\f{D}S$ approximately has the form of $u^\alpha$. 
If $m=1=\sigma$, then \eqref{eq:pp} becomes the classical Keller--Segel model, where it is known that all solutions exist globally and are bounded if $N=1$ (\cite{osaki_yagi}), for $N=2$, smallness of the initial mass $\io u(\cdot,0)$ is sufficient to guarantee boundedness (\cite{gajewski_zacharias,nagai_senba_yoshida_tmi}),  
whereas for large initial mass in $N=2$ (\cite{horstmann_wang,senba_suzuki}, cf. also \cite{herrero_velazquez_97}) %NOTE: herrero_velazquez_97 deals with a system with -1 instead of -v in the second equation
and any mass in $N\ge 3$ (\cite{win_aggregationvs}) there are initial data leading to unbounded solutions. 
It has also been shown that this blow-up occurs within finite time for a ``large'' set of (radially symmetric) initial data (\cite{win_buhigherdim,mizoguchi_winkler_13}).

% PP\\
% \subsection{$m=\sigma=1$}
% % \cite{gajewski_zacharias} $N=2$, $m=1=\sigma$, GE, bddness\\
% % \cite{nagai_senba_yoshida_tmi}: $\Om\subℝ^2$, $m=1$, $\sigma=1$, GEbd; Lem 3.3: LF.
% % 
% % \cite{horstmann_wang}: $m=\sigma=1$, $N=2$: ex soln which bu in finite \textbf{or} infinite time. 
% % 
% % \cite{senba_suzuki} $N=2$, $m=\sigma=1$: another proof of horstmann/wang; and type of bu: $8\pi\delta$ in origin (...) 
% 
% %\cite{herrero_velazquez_96}$N=2$, $m=\sigma=1$, formal Asymptotic expansions\cite{herrero_velazquez_97}:``rigorous proof of the former''
% %weglassen, nur 2d
% 
% % \cite{win_aggregationvs} 
% %$m=\sigma=1$ (i.e. $\alpha=-1$)
% 
% \cite{win_buhigherdim} $N\ge 3$, $m=1=\sigma$ finite time BU for a dense set of initial data.

%\subsection{$m=1$}
% \cite{horstmann_winkler}: $m=1$, $S\le u^\alpha$. $\alpha<\f2N$: GEbd; $S\ge u^\alpha$: $\alpha>\f2N$: there are unbdd radial solns if $\al>2$, ($\al\in(1,2)$ and $N\in\{2,3\}$) (and some upper growth cond on $S$)) or ($\al\in(\f2N,1)$ or $\al\in(\f2N,\f2{N-2})$ and upper growth cond)

Retaining linear diffusion ($m=1$) but varying $\sigma$, it turns out that $-\alpha<\f2N$ leads to global existence, but on the other hand if $-\alpha$ is slightly larger than $\f2N$ then there are unbounded radially symmetric solutions. \cite{horstmann_winkler} %results do not cover /every/ -alpha>\f2N, but at least an interval starting from \f2N (and every alpha with -α>2. 
If both diffusion and sensitivity are allowed to be nonlinear, it is the same condition distinguishing global existence from possible blow-up: 
If $-α>\f2N$ then there are unbounded solutions (\cite{win_volumefilling}), 
whereas complementarily \cite{tao_winkler} asserts global boundedness under the condition that $-α<\f2N$. (For analogous boundedness and blow-up results in a two-species model see \cite{tian_zheng}.)
%: $\f{S(u)}{D(u)}\le cu^{\alpha}$, $\alpha<\f2N$: global bddness (under merely technical algebraic growth condition on $D$).
Also, it was for models of this kind %achtung, abtrennen von degen. 
that the convexity assumption on domains often used in earlier works on chemotaxis models was removed in  \cite{ishida_seki_yokota}, where again $-α<\f2N$ is the condition ensuring global boundedness of the solutions. 
%\cite{ishida_seki_yokota} $-\alpha<\f2N$: global bounded solutions, without convexity assumption

% \cite{cieslak_08}: $S(u)=u\beta(u)$. $N=1$, $\f\beta D\le const$ bdd.; $n=2,3$ $\f\beta D\le const$m $\beta(u)\le Mu^{-\gamma}$, $\gamma>N-1$: GE, $\beta\le Mu^{-\gamma}$, ($\gamma>N-1$), $\f{u\beta(u)}{D(u)}\ge Mu^\gamma$, $\gamma>\f2N$: ex unbdd (rad) sol.; $n=2,3$, $\f{u\beta(u)}{D(u)}\le Mu^{\gamma}, \gamma<\f2N$: GE bdd. all assume: $D$ and $β$ bounded!

If $\Omega$ is 1-dimensional and $\sigma=1$, the solutions also remain bounded in the case $α=-2=-\f2N$ as has been shown very recently in \cite{cieslak-fujie-17-2}. 
%Im Eindimensionalen und für $\sigma=1$ sind die Lösungen übrigens auch für $α=-2=-\f2N$ beschränkt, as has been shown very recently in \cite{cieslak-fujie-17-2}
%\cite{cieslak-fujie-17-2} $N=1$, $\sigma=1$, $D(u)=\f1{1+u}$: global class sol ; $\al=-2=-\f2N$

In the presence of logistic source terms, one condition ensuring global existence again is $-α<\f2N$ -- another would be sufficient strength of the consumptive part of the logistic source (for precise conditions refer to \cite{xie_xiang,wang_mu_zheng_14jde}); blow-up results have not been obtained. 

% \cite{xie_xiang}: both pp and pe, with logistic source: GE+bdness if $\sigma<m+\f2N$ or damping part of log source strong enough (improves on Wang/Li/Mu 14 DCDS, \cite{wang_mu_zheng_14jde})

The case of degenerate diffusion ($D(u)=u^{m-1}$ instead of $D(u)=(u+1)^{m-1}$) requires additional technical care (and restrictions such as $m\ge1$), but finally the same conditions on $α$ are recovered, for boundedness (\cite{ishida_yokota_12,sugiyama07,sugiyama_kunii}) as well as for blow-up (\cite{ishida_ono_yokota, ishida_yokota}). \\

% \cite{ishida_ono_yokota}: $D(u)=u^{m-1}$, $S(u)=u^{q-1}$ (focus on degenerate case): if $q\ge m+\f2N$: every radial strong solution with low initial energy is unbounded. 
% 
% \cite{ishida_yokota}: $m\ge 1$, $q\ge 2$:  every ``energy solution'' ... (as above, but here existence is known)
% 
% \cite{ishida_yokota_12} $D(u)=u^{m-1}$, $S(u)=u^{q-1}$, $m\ge 1$, $q\ge 2$, $q<m+\f2N$: global weak sol. 
% (see here for results of \cite{sugiyama07}, \cite{sugiyama_kunii}, which needs stronger condition $q\le m$)

% GEresults general, also including log source: here, diffusion strong if compared to sensitivity $-α>\f2N$, alternatively damping by logistic source terms strong. 
% 
% possibly degenerate diffusion 

%was bis hierhin noch nicht unterschieden ist, ist finite/infinite time. 

With the exception of \cite{win_buhigherdim}, the works mentioned up to this point do not help in distinguishing blowup in finite time from that occuring after infinite time. Building on the method of \cite{win_buhigherdim}, in \cite{cieslak_stinner_12} Cieslak and Stinner showed that finite-time blowup occurs if $N\ge 3$: $σ\ge 1$, $m\in ℝ$, $-α>\f2N$. Results pertaining to 2-dimensional domains can be found in \cite{cieslak_stinner_14}. The more recent extension \cite{cieslak_stinner_15} of \cite{cieslak_stinner_12} showed finite-time blow-up if 
$-α>\f2N$, $m\ge 1$ $N\ge 3$ %showed finite-time blow-up for many initial data, 
or $σ>\f2N, -α>\f2N$ (and $m\ge1$ if $σ<1$) 
and infinite-time blow-up under the condition that $\f2N<-α<\f2N-σ$. Similarly, solutions blow-up after infinite time if $σ\le  0$ and $-α>\f2N$ (\cite{win_itbu}). These papers also show that blow-up occurs for ``many'' initial data. For a different class of diffusivity and sensitivity functions, consult  \cite{winkler_17}, which gives conditions ensuring blow-up in infinite time for $D$ and $S$ being of exponential type.\\

Another relative of \eqref{eq:firstsystem} is the \textbf{further simplified system} 
\begin{align}\label{eq:jl}
 u_t&=\nabla\cdot(D(u)\nabla u)-\nabla \cdot(S(u)\nabla v)\\
  0&= \Delta v -\io u_0 + u.\nn
\end{align}
It has the convenient property that the analysis can be performed on a single scalar parabolic equation for the cumulated mass $w(r,t)=\int_0^r ρu(ρ,t) dρ$, which -- in contrast to \eqref{eq:firstsystem} -- is accessible for comparison arguments.  
For the classical case of 2-dimensional domains, $m=\sigma=1$, Jäger and Luckhaus (\cite{jaeger_luckhaus}) thereby showed existence of radially symmetric initial data such that $u$ explodes in the center of the domain after finite time. On the other hand, solutions rising from initial data with small mass exist globally  (\cite{calvez_carrillo}).

If $\sigma=1$ and $m$ is such that $-α<\f2N$, then solutions exist globally and are bounded, whereas $-α>\f2N$ may incur blow-up within finite time (\cite{cieslak_winkler_08}). If $\sigma$ may vary and $m\le 1$, \cite{winkler_djie} again asserts boundedness under the condition that $-α<\f2N$ and the possibility of blow-up if $-α>\f2N$, provided $σ>0$. 

%\cite{cieslak_winkler_08} $S(u)=u$, $D(u)\ge (1+u)^{-p}$. $p<\f2N-1$ gebdd; $D\le (1+u)^{-p}$, $p>\f2N-1$: finite time bu (also 1d!)

% The characterization of $S$ and $D$ for which bu can be achieved in \cite{winkler_djie} again yields 
% : $-α<\f2n$ GE, bdd.; $-α>\f2n$: finitetime BU \red{positivity of $\sigma$??? in welchen Dimensionen (auch $1$?)} \\
Albeit only in higher dimensions ($N\ge 5$) and for small $k>1$, in \eqref{eq:jl} with additional source $+u-μu^k$, finite-time blow-up was shown despite the logistic growth restriction in \cite{zheng_mu_hu} by extension of \cite{win_blowup_higherdim_despite_loggrowth} if $m\in(0,2-\f2N)$ and $σ\in(1,\f{2+2m}3)$ % and if $k$ is sufficiently small 
(\cite[Theorem 1.2]{zheng_mu_hu}). Again, largeness of $μ$ and $k$ or $-α<\f2N$ ensure global existence. \\

%\cite{jaeger_luckhaus} $N=2$ smallmass GE, largemass: ex rad init data st explosion of u in center at a finite time

%\cite{zheng_mu_hu} log source (hence $-m(t)$): GEbd if $\sigma<m+\frac2N$, $k>1$ ($-\mu u^{k}$) oder $\f\mu\chi$ groß und $k\ge \sigma+1$; if $N\ge 5$, $k$ small: finite time bu (like \cite{win_blowup_higherdim_despite_loggrowth} )

%\subsection{pe} 

As to the \textbf{parabolic-elliptic system \eqref{eq:firstsystem}}, the only available blow-up results deal with the classical model with $m=σ=1$, where finite-time blow-up has been shown to occur in two-dimensional domains for radial initial data with sufficiently large mass that are concentrated in the sense that their second moment is small \cite{nagai_senba_98,nagai01}, or in higher dimensional domains, where a higher moment seemed decisive \cite{nagai_senba_98} (for a corresponding result in $\Om=ℝ^N$ and condition on the second moment see \cite{calvez_corrias_ebde}). 

On the other hand, for other choices of $m$ and $σ$, $-α<\f2N$ again ensures global boundedness (see \cite[Thm 5.3]{cieslak_moralesrodrigo} for $N=3$, \cite{xie_xiang,zheng} for a general system also including logistic source terms, or \cite{wang_16} for a closely related parabolic-elliptic-elliptic attraction-repulsion system). \\

\textbf{Results.} 
For functions 
\begin{equation}\label{eq:DSpower}
 D(u)=(u+1)^{m-1}, \qquad S(u)=u(u+1)^{σ-1}
\end{equation}
in 
\begin{subequations}\label{eq:system}
\begin{align}
 u_t&=\nabla\cdot(D(u)\nabla u) - \nabla\cdot(S(u)\nabla v)\label{ueq} &&\qquad \text{in } \Om\times(0,\Tmax)\\
 0&= \Delta v - v +u \label{veq} &&\qquad \text{in } \Om\times(0,\Tmax)\\
 &u(\cdot,0)=u_0 &&\qquad \text{in } \Om \nn\\
 &\delny u\bdry=\delny v\bdry=0 &&\qquad \text{in } (0,\Tmax)\nn
\end{align}
\end{subequations}
we will attempt to characterize, which exponents spawn which kind of solution behaviour. Slightly generalizing $D$ and $S$ if compared to \eqref{eq:DSpower}, we will assume that 
\begin{equation}\label{condonDandS}
 D,S \in C^1(ℝ) \; \text{are such that } D(ξ)>0 \text{ for all } ξ\ge 0, \; S(ξ)>0\quad \text{for all } ξ>0, 
\end{equation}
and will usually assume that, in addition, with $c_D>0$, $C_S>0$, 
\begin{equation}\label{Dgeq}
 D(u)\ge c_D(u+1)^{m-1} \qquad \text{for } u>0,
\end{equation}
and
\begin{equation}\label{Sleq}
 S(u)\le C_S u(u+1)^{σ-1} \qquad \text{for } u>0.
\end{equation}
Defining 
\begin{equation}\label{defG}
 G(u):= \int_1^u\int_1^s \f{D(ξ)}{S(ξ)} dξds, \qquad u\in(0,∞), 
\end{equation}
we will furthermore assume 
\begin{equation}\label{Gleq}
 G(ζ)\le C_G(1+ζ^{2+α})
\end{equation}
for some $α\inℝ$ and all $ζ>0$ 
(which is consistent with the assumption that $\f DS\approx u^{\alpha}$ from the first part of the introduction and, in the case of \eqref{eq:DSpower} is satisfied with $α=m-σ-1$).

Our first result will then be to recover the conditions for global existence and boundedness of solutions:
\begin{theorem}\label{thm:bdness} Let $\Om\subsetℝ^N$, $N\ge 1$, be a bounded domain with smooth boundary. 
 Let $\sigma\inℝ$ and $m\inℝ$ satisfy
\begin{equation}
  σ<m-\f{N-2}N \label{eq:sigmasmaller.intro}
\end{equation}
and assume that $D$ and $S$ fulfil \eqref{condonDandS} as well as \eqref{Dgeq} and \eqref{Sleq} with some $c_D>0$ and $C_S>0$. Then for every $β\in(0,1)$ and every nonnegative function $u_0\in C^{β}(\Ombar)$ the solution $(u,v)$ to \eqref{eq:system} exists globally and is bounded. 
\end{theorem}

This will be the consequence of a differential inequality for $\io u^p$, a small change in which can also be used to show global existence of solutions for nonpositive $σ$: 
\begin{theorem}\label{thm:GEifsigmanegative} 
Let $\Om\subsetℝ^N$, $N\ge 1$, be a bounded domain with smooth boundary.
 Let $m\in ℝ$, $σ\le 0$, $β\in(0,1)$. Assume that $D$, $S$ satisfy \eqref{condonDandS} and \eqref{Dgeq}, \eqref{Sleq} with some $c_D>0$ and $c_S>0$. Then for every $0\le u_0\in C^{β}(\Ombar)$, the solution $(u,v)$ to \eqref{eq:system} exists globally. 
\end{theorem}

The most exciting part, however, will be the detection of unbounded solutions. Here we will rely on 
\begin{equation}\label{defF}
 \calF(u,v):=\f12\io|\na v|^2 + \f12\io v^2 -\io uv +\io G(u),\qquad (u,v)\in C^0(\Ombar)\times C^1(\Ombar).
\end{equation}

which has been known to be an energy functional for \eqref{eq:pp} and \eqref{eq:firstsystem} for a long time (see \cite{nagai_senba_yoshida_ge,gajewski_zacharias,biler,win_buhigherdim}) 
and lies at the core of unboundedness results in the parabolic-parabolic setting (\cite{horstmann_winkler,win_buhigherdim,cieslak_stinner_12,cieslak_stinner_14,cieslak_stinner_15}, see above), where it is known that initial data $(u_0,v_0)$ with sufficiently negative energy $\calF(u_0,v_0)$ yield unbounded solutions, if $S$ and $D$ satisfy 
\begin{equation}\label{eq:winvolumefilling2.11}
 \int_{s_0}^s \f{τD(τ)}{S(τ)} dτ\le \f{N-2-δ}N \int_{s_0}^s\int_{s_0}^{σ} \f{D(τ)}{S(τ)} dτdσ + Ks \quad \text{for all } s\ge s_0,
\end{equation}
with some $δ>0$, $s_0\ge 1$ and $K\ge 0$. 

\begin{remark}
 Condition \eqref{eq:winvolumefilling2.11} is, in particular, satisfied if $u^{β}\f{D(u) }{S(u)} \to c_0>0$ as $u\to \infty$ for some $β>\f2N$ (\cite[Cor. 5.2]{win_volumefilling}). If $D(u)=(u+1)^{m-1}$ and $S(u)=u(u+1)^{σ-1}$, then $β=-α$. 
\end{remark}

In stark contrast to the parabolic-parabolic case, in our search for suitable initial data, we will have to ensure that $u_0$ and $v_0$ ``fit''. (Since no initial data for $v$ are part of \eqref{eq:system}, we have to define $v_0$ by $0=\Delta v_0-v_0+u_0$, but are at least justified in using these functions by Lemma \ref{lem:unbd}.) The corresponding construction will be what Section \ref{sec:construct} will be devoted to. 

Not satisfied with having found one function $u_0$ that leads to blow-up, we will then proceed to show that there are actually ``many'' choices of initial data with this property:

\begin{theorem}\label{thm:unbounded}
 Let $\Om=B_R\subset ℝ^N$, $N\ge 3$. Let $S$, $D$ be such that \eqref{condonDandS}, $S(0)=0$ and \eqref{eq:winvolumefilling2.11} with some $s_0>0$, $K>0$, $δ>0$ are satisfied and that $G$ as defined in \eqref{defG} satisfies \eqref{Gleq} with some $α\inℝ$ and $C_G>0$.  If 
%\begin{equation}\label{alge2durchn}
\( -α>\f2N\), 
%\end{equation}
the following holds:\\
 Let $p\in[1,\f{2N}{N+2})$ if $α\le -\f4{N+2}$ and $p\in[1,-\f{αN}2)$ if $α> -\f4{N+2}$. 
 Given radially symmetric $u_0\in C^{β}(\Ombar)$ for some $β>0$, there are radially symmetric functions $\ue, \ve$ such that $0=\Delta \ve-\ve +\ue$, $\delny\ve\bdry=0$ for any $η\in(0,1)$, and 
\[
 \norm[\Lom p]{\ue-u_0} \to 0 \qquad \text{as } η\searrow 0.
\]
 and that the solutions to \eqref{eq:system} for these initial data $u(\cdot,0):=\ue$ blow up. 
\end{theorem}

A combination of Theorem \ref{thm:GEifsigmanegative} and Theorem \ref{thm:unbounded} in particular entails 
\begin{corollary}
 Let $\Om=B_R\subset ℝ^N$, $N\ge 3$, let $S$, $D$ satisfy \eqref{condonDandS}, $S(0)=0$ and \eqref{eq:winvolumefilling2.11} with some $s_0>0$, $K>0$, $δ>0$ as well as \eqref{Sleq} and \eqref{Dgeq} with some $c_D>0$, $C_S>0$ and $m\in ℝ$ and $σ\le 0$. Assume that $G$ as defined in \eqref{defG} satisfies \eqref{Gleq} with some $α\inℝ$ and $C_G>0$.  If $-α>\f2N$, let $p\in[1,\f{2N}{N+2})$ if $α\le -\f4{N+2}$ and $p\in[1,-\f{αN}2)$ if $α> -\f4{N+2}$. Given radially symmetric $u_0\in C^{β}(\Ombar)$ for some $β>0$, there are radially symmetric functions $\ue, \ve$ such that $0=\Delta \ve-\ve +\ue$, $\delny\ve\bdry=0$ for any $η\in(0,1)$, and 
\[
 \norm[\Lom p]{\ue-u_0} \to 0 \qquad \text{as } η\searrow 0.
\]
 and that the solutions to \eqref{eq:system} for these initial data $u(\cdot,0):=\ue$ exist globally, but blow up at time $\infty$. 
\end{corollary}

In particular, with this we have detected a wide range of parameters $m$, $σ$ for which infinite-time blow-up is, in some sense, the typical behaviour of radially symmetric solutions to \eqref{eq:system}. 

%\red{While in line with \cite{win_itbu} (and its predecessor \cite{} cieslak stinner, for smaller range of parameters), gibt's das eigentlich nicht so oft. Siehe aber VF-models}

% \begin{equation}\label{DdSgeq}
%  \f{D}{S} \ge (u+1)^{m-σ-1}, \qquad α=m-σ-1
% \end{equation}

\section{Global existence and boundedness}

This section is devoted to the results on global existence and boundedness. We begin the preparations by recalling a statement on local existence including an extensibility criterion. A similar result can be found, for example, in \cite[Lemma 2.1]{xie_xiang}. Note, however, that the present lemma gives a stronger assertion concerning the regularity of $v$ at time $0$, which will be crucial for our purpose. 

\begin{lemma}\label{lem:locex}
Let $S,D\in C^1(ℝ)$ be such that $D(s)>0$ for all $s\ge0$, let $β\in(0,1)$. Then for any nonnegative $u_0\in C^{β}(\Ombar)$ there is $\Tmax>0$ and a unique pair of functions $(u,v)\in (C^0(\Ombar\times[0,\Tmax))\cap C^{2,1}(\Ombar\times(0,\Tmax)))\times (C^0([0,\Tmax),C^1(\Ombar))\cap C^{2,1}(\Ombar\times(0,\Tmax)))$
 (hereafter: ``classical solution'') that satisfies \eqref{eq:system} and is such that 
\begin{equation}\label{extcrit}
 \text{either } \Tmax=\infty \; \text{or } \limsup_{t\nearrow \Tmax} \norm[\infty]{u(\cdot,t)} = \infty. 
\end{equation}
Moreover, $u$ and $v$ are nonnegative in $\Ombar\times(0,\Tmax)$.

% \red{tbd} 
% 
%  Let $S$ and $D$ satisfy \eqref{condonDandS} and \eqref{Dgeq} for some $m\inℝ$ and $c_D>0$, let $β\in(0,1)$. Then for any $u_0\in C^{β}(\Ombar)$ there are $γ\in(0,α]$, $\Tmax$ and a unique pair of solutions $(u,v)\in C^{2+γ,1+\f{γ}2}(\Ombar\times(0,\Tmax))\cap C^{γ,\f{γ}2}(\Ombar\times[0,\Tmax))\times C^{2+γ,2+\f{γ}2}(\Ombar\times[0,\Tmax))$ to \eqref{eq:system} such that 
% \begin{equation}\label{extcrit}
%  \text{either } \Tmax=\infty \; \text{or } \limsup_{t\nearrow \Tmax} \norm[\infty]{u(\cdot,t)} = \infty. 
% \end{equation}
\end{lemma}
\begin{proof}
 We begin the proof with the assertion on uniqueness and assume that, for some fixed $T>0$, $(u_1,v_1),(u_2,v_2)\in (C^0(\Ombar\times[0,T))\cap C^{2,1}(\Ombar\times(0,T))\times (C^0([0,T),C^1(\Ombar))\cap C^{2,1}(\Ombar\times(0,T)))$ both solve \eqref{eq:system} with the same nonnegative initial data $u_1(\cdot,0)=u_0=u_2(\cdot,0)$. We note that this also implies $v_1(\cdot,0)=v_2(\cdot,0)$, because these functions solve $0=\Delta v_i(\cdot,0) - v_i(\cdot,0) +u_0$ in a weak sense due to $v_i\in C^0([0,T),C^1(\Ombar))$ and \eqref{veq}, and the weak solution of this equation is unique.

We pick an arbitrary $T'\in(0,T)$ and let $c_1>0$, $c_2>0$, $c_3>0$, $c_4>0$ and $c_5>0$ be such that 
\begin{align*}
 0&\le u_i(x,t)\le c_1 \quad \text{for all } (x,t)\in\Om\times(0,T'), i\in\{1,2\},\hspace{-2cm}\\
 &\sup_{ξ\in(0,c_1)} D(ξ)\le c_2, & 
 &\norm[\Lom\infty]{\na v_1(\cdot,t)} \le c_3 \quad \text{for all } t\in(0,T'),\\
 &\sup_{ξ\in(0,c_1)} |S'(ξ)|\le c_4, &
 &S(u_2(x,t))\le c_5 \quad \text{for all} (x,t)\in\Om\times(0,T').
\end{align*}

We have that 
\[
 0=\Delta (v_1-v_2) - (v_1-v_2) + (u_1-u_2)\qquad\text{in } \Om\times(0,T')
\]
 and hence obtain 
\begin{align*}
 \f12 \f{d}{dt} \kl{\io |\na (v_1-v_2)|^2 + \io (v_1-v_2)^2} &= \io \na (v_1-v_2)_t \na (v_1-v_2) + \io (v_1-v_2)_t (v_1-v_2)\\
 &= \io (-\Delta (v_1-v_2) + (v_1-v_2))_t (v_1-v_2) \\
 &= \io (u_1-u_2)_t(v_1-v_2) \qquad \qquad \text{in } (0,T'). 
\end{align*}
If we introduce $\Dbar(s):=\int_0^s D(ξ) dξ$ and insert \eqref{ueq}, we end up with 
\begin{align}\label{eq:uniqueness:ddt}
 \f12 \f{d}{dt}& \kl{\io |\na (v_1-v_2)|^2 + \io (v_1-v_2)^2} \nn\\
&= -\io \na(\Dbar(u_1)-\Dbar(u_2))\na (v_1-v_2) + \io (S(u_1)\na v_1 - S(u_2)\na v_2)\na (v_1-v_2)
\end{align}
in $(0,T')$. 
By the mean value theorem and the condition that $D(s)\ge c_6:=\inf_{ξ\in(0,c_1)} D(ξ) >0$, we have that $(\Dbar(u_1)-\Dbar(u_2))(u_1-u_2) \ge c_6 (u_1-u_2)^2$ and that $(\Dbar(u_1)-\Dbar(u_2))^2\le c_2^2(u_1-u_2)^2$. Integration by parts, \eqref{veq} and Young's inequality show that in $(0,T')$
\begin{align*}
 -\io &\na(\Dbar(u_1)-\Dbar(u_2))\na (v_1-v_2)\\
 &= \io (\Dbar(u_1)-\Dbar(u_2))\Delta (v_1-v_2)\\
 &= \io (\Dbar(u_1)-\Dbar(u_2))(v_1-v_2) - \io (\Dbar(u_1)-\Dbar(u_2))(u_1-u_2)\\
 &\le  \f{c_2^2}{2 c_6} \io (v_1-v_2)^2 + \f{c_6}{2c_2^2} \io (\Dbar(u_1)-\Dbar(u_2))^2 - \io (\Dbar(u_1)-\Dbar(u_2))(u_1-u_2)\\
 &\le \f{c_2^2}{2 c_6} \io (v_1-v_2)^2 + \f{c_2^2c_6}{2c_2^2} \io (u_1-u_2)^2 - c_6\io (u_1-u_2)^2\\
 &= \f{c_2^2}{2 c_6} \io (v_1-v_2)^2 - \f{c_6}2\io (u_1-u_2)^2, 
\end{align*}
whereas the last term in \eqref{eq:uniqueness:ddt} can be estimated according to 
\begin{align*}
 \io& (S(u_1)\na v_1 - S(u_2)\na v_2)\na (v_1-v_2) \\
 &= \io \kl{S(u_1)\na v_1- S(u_2)\na v_1 + S(u_2)\na(v_1-v_2)}\na(v_1-v_2)\\
&\le c_3 \io |S(u_1)-S(u_2)| |\na (v_1-v_2)| + c_5 \io |\na(v_1-v_2)|^2\\
&\le c_3c_4 \io |u_1-u_2| |\na (v_1-v_2)| + c_5 \io |\na(v_1-v_2)|^2\\
&\le \f{c_6}2 \io |u_1-u_2|^2 + \left(\f{c_3^2c_4^2}{2c_6}+c_5\right) \io |\na(v_1-v_2)|^2 \qquad \text{in } (0,T'). 
\end{align*}
In conclusion, in $(0,T')$ we obtain 
\[
 \f12 \f{d}{dt} \kl{\io |\na (v_1-v_2)|^2 + \io (v_1-v_2)^2} \le \left(\f{c_2^2}{2c_6}+\f{c_3^2c_4^2}{2c_6}+c_5\right) \kl{\io |\na(v_1-v_2)|^2 + \io (v_1-v_2)^2}, 
\]
which by Grönwall's inequality and $v_1(\cdot,0)=v_2(\cdot,0)$ shows that $v_1=v_2$ in $\Om\times(0,T')$ and hence in $\Om\times(0,T)$ by arbitrarity of $T'$. By \eqref{veq}, this entails that $u_1=u_2$.

For sufficiently small $T>0$ (where the precise meaning of ``sufficiently small'' depends on $\norm[\Lom\infty]{u_0}$ and $\norm[C^{β}(\Ombar)]{u_0}$), the map $\calS$ defined on the set 
\[
 X:=\{u\in C^0(\Ombar\times[0,T]) \mid \norm[\Lom\infty]{u}\le \norm[\Lom\infty]{u_0}+1, u(\cdot,0)=u_0 \}
\]
by $\calS \uhat=\utilde$, with $\utilde$ being the solution of 
\[
 \utilde_t=\na\cdot(D(\uhat)\na \utilde - S(\uhat)\na \vhat), \qquad \delny \utilde\bdry=0, \quad \utilde(\cdot,0)=u_0, 
\]
where $\vhat$ solves 
\begin{equation}\label{eq:ex:vhateq}
0=\Delta \vhat-\vhat+\uhat, \qquad \delny \vhat\bdry=0,
\end{equation}
can be seen to be a continuous and compact map of $X$ into $X$ and to hence have a fixed point $u$ according to Schauder's theorem. 
The corresponding calculations rely on the well-known elliptic regularity estimate for any $p\in(1,\infty)$ asserting the existence of a constant $C_p>0$ such that all solutions of \eqref{eq:ex:vhateq} satisfy 
\begin{equation}\label{eq:ellreg}
 \norm[W^{2,p}(\Om)]{\vhat}\le C_p\norm[\Lom p]{\uhat}
\end{equation}
(which can, e.g. be obtained from \cite[Thm. 19.1]{friedman} in combination with the estimate $\norm[\Lom p]{\vhat}\le \norm[\Lom p]{\uhat}$ that results from \eqref{eq:ex:vhateq} by testing with (an approximation of) $v^{p-1}$) and on parabolic regularity statements that can be found in \cite[Lemma 2.1]{lankeit_singcon}, parts iii) and iv), which also guarantee $u\in C^0(\Ombar\times[0,T])\cap C^{2,1}(\Ombar\times(0,T))$. We let $v$ be the solution of \eqref{eq:ex:vhateq} for $\uhat=u$. As particular consequence of \eqref{eq:ellreg} applied to some $p>N$ and linearity of \eqref{eq:ex:vhateq} let us note that 
\[
 \normm{C^0([0,T);C^1(\Ombar))}{v} \le c_6 \normm{C^0([0,T);W^{2,p}(\Om))}{v}\le C_pc_6 \normm{C^0([0,T);\Lom p)}{u}
\]
and hence $v\in C^0([0,T),C^1(\Ombar))$. The extensibility criterion \eqref{extcrit} can be concluded from the dependence of $T$ on $\norm[\Lom\infty]{u_0}$ and $\norm[C^\beta(\Ombar)]{u_0}$ in combination with \cite[Lemma 2.1 iv)]{lankeit_singcon} prohibiting blow-up of $\normm{C^\beta(\Ombar)}{u}$ while $\norm[\Lom \infty]{u}$ remains bounded. Nonnegativity is obtained from classical comparison theorems.
\end{proof}

In order to show boundedness of $u$, it suffices to estimate the norm of $u$ in a suitable $\Lom p$-space, with some large, but finite $p$. 
\begin{lemma}\label{lem:pimpliesinfty}
 Let $q_1>N+2$ and $q_2>\f{N+2}2$, $m\in ℝ$, $σ\in ℝ$, $β\in(0,1)$ and assume that $S$ and $D$ satisfy \eqref{condonDandS}, \eqref{Dgeq} and \eqref{Sleq} and let $u_0\in C^{β}(\Ombar)$. Let 
\[
 p>\max\left\{N,\;\f N2(1-m),\;q_1σ,\;1-m\f{(N+1)q_1-(N+2)}{q_1-(N+2)},\;1-\f{m}{1-\f{Nq_2}{(N+2)(q_2-1)}}\right\}. 
\]
 Then for every $K>0$ there is $C>0$ such that whenever $(u,v)\in $ solves \eqref{eq:system} in $\Om\times(0,T)$ for some $T>0$ and satisfies 
\[
 \norm[\Lom p]{u(\cdot,t)}\le K \text{ for all }t\in(0,T), 
\]
 then 
\[
 \norm[\Lom \infty]{u(\cdot,t)}\le C \text{ for all } t\in(0,T). 
\]
\end{lemma}
\begin{proof}
 Since $p>N$, $\norm[W^{1,\infty}(\Om)]{v(\cdot,t)}$ can be controlled by $\norm[\Lom p]{u(\cdot,t)}$ for $t\in (0,T)$ by elliptic regularity estimates (cf. \eqref{eq:ellreg}). With $f:=S(u)\na v$ and $g=0$ we hence have that $\norm[\Lom{q_1}]{f}\le C_S\norm[\Lom\infty]{\na v} \norm[\Lom{q_1}]{(u+1)^{σ}}$ is bounded in $(0,T)$ and \cite[Lemma A.1]{tao_winkler} is applicable.
\end{proof}

According to the previous lemma and \eqref{extcrit}, global existence and boundedness can be shown by ensuring that $\io u^p$ is bounded locally or globally in time, respectively, for some large $p$. These assertions will rest on the following differential inequality.

\begin{lemma}\label{lem:diffineq}
Let $m\in ℝ$, $σ\in ℝ$, $β\in(0,1)$. 
Let $D$ and $S$ satisfy \eqref{condonDandS} as well as \eqref{Dgeq} and \eqref{Sleq} with some $c_D>0$ and $C_S>0$ and let $u_0\in C^{β}(\Ombar)$ be nonnegative. 
Let $p\in(1,\infty)$ satisfy $p>1-σ$. Then the solution $(u,v)$ of \eqref{eq:system} satisfies 
\begin{equation}\label{diffineq}
 \f1p \f{d}{dt} \io (u+1)^p\le -\f{4c_D(p-1)}{(m+p-1)^2}\io |\na (u+1)^{\f{m+p-1}2}|^2 + \f{C_S(p-1)}{p+σ-1}\io (u+1)^{p+σ}
\end{equation}
in $(0,\Tmax)$.
\end{lemma}
\begin{proof}
We introduce 
\[
 \Stilde(u)=\int_0^u (ζ+1)^{p-2}S(ζ) dζ \qquad \text{for } u\ge 0
\]
and note that according to \eqref{Sleq}
\begin{equation}\label{Stildeleq}
 \Stilde(u) \le C_S \int_0^u(ζ+1)^{p+σ-2} dζ \le \f{C_S}{p+σ-1}(u+1)^{p+σ-1} \quad \text{for any } u\ge 0.
\end{equation}
We then use the first equation of \eqref{eq:system} together with integration by parts and the estimates \eqref{Dgeq} and $-\Delta v=u-v \le u$ to obtain 
\begin{align*}
 \f1p \f{d}{dt} \io (u+1)^p &= \io (u+1)^{p-1} \na\cdot(D(u)\na u - S(u)\na v)\\
 &= -(p-1)\io D(u)(u+1)^{p-2} |\na u|^2 + (p-1)\io (u+1)^{p-2} S(u)\na u\cdot \na v\\
 &\le  -c_D(p-1)\io (u+1)^{m+p-3} |\na u|^2 + (p-1)\io \na \Stilde(u)\cdot\na v\\
 &= -\f{4c_D(p-1)}{(m+p-1)^2}\io |\na (u+1)^{\f{m+p-1}2}|^2 - (p-1)\io \Stilde(u) Δv\\
&\le -\f{4c_D(p-1)}{(m+p-1)^2}\io |\na (u+1)^{\f{m+p-1}2}|^2 + (p-1)\io u\Stilde(u)% \\
%&\le -\f{4c_D(p-1)}{(m+p-1)^2}\io |\na (u+1)^{\f{m+p-1}2}|^2 + \f{C_S(p-1)}{p+σ-1}\io (u+1)^{p+σ}
\end{align*}
in $(0,\Tmax)$, which due to \eqref{Stildeleq} results in \eqref{diffineq}.
\end{proof}

If $σ$ is negative, global existence directly results from the differential inequality \eqref{diffineq}. 

\begin{proof}[Proof of Theorem \ref{thm:GEifsigmanegative}]
 Local existence of a solution $(u,v)$ is ensured by Lemma \ref{lem:locex}. 
 Letting $p>1-σ$ be so large that Lemma \ref{lem:pimpliesinfty} becomes applicable, we employ Young's inequality to find $c_1>0$ such that 
\[
 \io (u+1)^{p+σ}(\cdot,t) \le c_1 + \io (u+1)^p (\cdot,t) =:y(t) \quad \text{for all } t\in(0,\Tmax), 
\]
so that Lemma \ref{lem:diffineq} guarantees $y'(t)\le y(t)$ for all $t\in (0,\Tmax)$. A combination of Lemma \ref{lem:pimpliesinfty} and \eqref{extcrit} then results in global existence. 
\end{proof}

Apparently, the estimate underlying this proof of Theorem \ref{thm:GEifsigmanegative} is rather rough, even neglecting the dissipative term in \eqref{diffineq}. If $σ<m-\f{N-2}N$, better estimates can be achieved, finally leading to boundedness of solutions, regardless of the sign of $σ$. We begin the preparation of the corresponding proof with the following different estimate of $\io (u+1)^{p+σ}$.

\begin{lemma}\label{lem:upsigma}
Let $\Om\subsetℝ^N$, $N\ge 1$, be a bounded domain with smooth boundary.
Let $\sigma\inℝ$ and $m\inℝ$ satisfy 
\begin{equation}
  σ<m-\f{N-2}N. \label{eq:sigmasmaller}
\end{equation}
Let $p>\max\{1,2-m,1-σ,(\f N2-1)σ+\f N2(1-m)\}$ and $c_D>0$. 
Then for any $K>0$ 
there is $C>0$ such that every nonnegative function $w\in C^2(\Ombar)$ which satisfies 
\begin{equation}\label{wl1bd}
 \io (w+1) \le K
\end{equation}
fulfils 
\[
 \f{C_S(p-1)}{p+σ-1}\io (w+1)^{p+σ} \le \f{2c_D(p-1)}{(m+p-1)^2} \io |\nabla (w+1)^{\f{m+p-1}2}|^2 + C.
\]
\end{lemma}

\begin{proof}

We let 
\[
 a:= \f{\f {N}2 (p+m-1)(1-\f1{p+σ})}{1+\f {N}2(p+m-2)},
\]
so that $-\f{N(p+m-1)}{2(p+σ)}=(1-\f {N}2)a - \f{N(p+m-1)}2(1-a)$. By the conditions on $p$, positivity of $a$ is obvious. Moreover, we have $m-1>-\f2{N} p+(1-\f2{N})σ$, which means $p+m-1>p+σ-\f{2p}{N} -\f{2σ}{N}=(1-\f2{N})(p+σ)$ and hence $-\f{p+m-1}{p+σ}<\f2N-1$. Therefore, $p+m-1-\f{p+m-1}{p+σ}<\f2{N}+p+m-2$, i.e. $\f {N}2(p+m-1)(1-\f1{p+σ})<1+\f {N}2(p+m-2)$, which shows that $a<1$. 
% any choice of $p>2-m$, $p>1-σ$ makes positivity of $a$ obvious. 
% 
% wlog: $p>(\f n2-1)σ + \f n2(1-m)$. Then 
% \begin{align*}
%  \f{2p}{n} &> (1-\f2n)σ+1-m\\
%  p+m-1&>p+σ-\f{2p}n-\f{2σ}n\\
%  p+m-1 &> (1-\f2n)(p+σ)\\
%  -\f{p+m-1}{p+σ}&<\f2n -1\\
%  p+m-1-\f{p+m-1}{p+σ}&<\f2n+p+m-2\\
%  \f n2(p+m-1)(1-\f1{p+σ})&<1+\f n2(p+m-2),
% \end{align*}
% and hence $a<1$. 

We thus can apply the Gagliardo--Nirenberg inequality to find $c_1>0$ such that 
\begin{align}\label{gni-appl1}
 \io (w+1)^{p+σ} &= \io (w+1)^{\f{p+m-1}2\cdot \f{2(p+σ)}{p+m-1}}\nn\\
 & = \norm[\f{2(p+σ)}{p+m-1}]{(w+1)^{\f{p+m-1}2}}^{\f{2(p+σ)}{p+m-1}}\nn\\
  &\le c_1\norm[\Lom2]{\na w^{\f{p+m-1}2}}^{\f{2(p+σ)}{p+m-1}a} \norm[\f{2}{p+m-1}]{(w+1)^{\f{p+m-1}2}}^{\f{2(p+σ)}{p+m-1}(1-a)} + c_1\norm[\f{2}{p+m-1}]{(w+1)^{\f{p+m-1}2}}^{\f{2(p+σ)}{p+m-1}}
\end{align}

% ($p>2-m$): 
% Furthermore, 
% \begin{align}
%  σ-m&<\f{2-N}n\nn \\
%  \f n2(σ-m+1)&<1\nn \\
%  \f n2(p+σ-1)&<1+\f n2(p+m-2)\nn \\
%  \f{\f n2(p+σ-1)}{1+\f n2(p+m-2)}&<1\nn\\
%  \f{p+σ}{p+m-1}a&<1\nn
% \end{align}

Furthermore, \eqref{eq:sigmasmaller} entails $\f {N}2(σ-m+1)<1$, so that $\f {N}2(p+σ-1)<1+\f {N}2(p+m-2)$ and hence 
\[
 \f{p+σ}{p+m-1}a=\f{\f {N}2(p+σ-1)}{1+\f {N}2(p+m-2)}<1. 
\]
We can therefore apply Young's inequality to \eqref{gni-appl1}
and accounting for \eqref{wl1bd}, we obtain $C=C(K)>0$ such that 
\[
 \f{C_S(p-1)}{p+σ-1}\io (w+1)^{p+σ} \le \f{2(p-1)c_D}{(m+p-1)^2} \norm[\Lom2]{\na w^{\f{p+m-1}2}}^{2}+C.\qedhere
\]
% \begin{align*}
%  \f{n(p+m-1)}2(1-\f1{p+σ})=(1-\f n2+\f n2(p+m-1))a\\
%  -\f{n(p+m-1)}{2(p+σ)}=(1-\f n2)a - \f{n(p+m-1)}2(1-a)
% \end{align*}
\end{proof}

We have seen that under the condition \eqref{eq:sigmasmaller} it is possible to estimate $\io (u+1)^{p+σ}$ by $\io |\na (u+1)^{\f{m+p-1}2}|^2$ and a constant.
This would transform \eqref{diffineq} into a statement of the form $y'\le c_1 - c_2\io |\na (u+1)^{\f{m+p-1}2}|^2$. 
 In order to derive boundedness of $\io (u+1)^p$ from this, we shall also require control of $\io u^p$ by means of $\io  |\na (u+1)^{\f{m+p-1}2}|^2$. If $σ\ge 0$, clearly the statement of Lemma \ref{lem:upsigma} is even stronger than that. Since our interest in this paper mainly lies in the case of $σ<0$, we prepare the following 

\begin{lemma}\label{lem:upgradient}
Let $\Om\subsetℝ^N$, $N\ge 1$, be a bounded domain with smooth boundary.
Let $m\in ℝ$, $c_D>0$ and $p>\max\{1,2-m,\f {N}2 (1-m)\}$. Then for every $K>0$ there is $C>0$ such that 
every nonnegative function $w\in C^2(\Ombar)$ satisfying $\io (w+1)\le K$ fulfils 

\begin{equation}\label{estupbynau}
 \kl{\io (w+1)^p}^{\f{m+p-1}p}\le \f{2c_D(p-1)}{(m+p-1)^2} \io |\na (w+1)^{\f{m+p-1}2}|^2 + C.
\end{equation}
\end{lemma}
\begin{proof}
We let 
\[
 b=\f{\f{{N}(m+p-1)}2(1-\f1p)}{1+\f {N}2(m+p-2)}
\]
so that $-\f{{N}(m+p-1)}{2p} = (1-\f {N}2)b - \f{{N}(m+p-1)}2(1-b)$ and that, by the conditions imposed on $p$, $b$ is clearly positive and 
\[
-\f1p(m+p-1) = \f{1-m}p-1<\f2N-1=\f2{N}(1+\f {N}2(m+p-2-(m+p-1))), 
\]
showing that $\f {N}2(m+p-1)(1-\f1p)<1+\f {N}2(m+p-2)$ and hence also $b<1$. 
From the Gagliardo--Nirenberg inequality we then obtain $c_1>0$ such that 
\begin{align*}
 \kl {\io (w+1)^p}^{\f{m+p-1}p} &= \kl{\io (w+1)^{\f{m+p-1}2 \cdot \f{2p}{m+p-1}}}^{\f{m+p-1}p}\\
 &= \norm[{\f{2p}{m+p-1}}]{(w+1)^{\f{m+p-1}2}}^2\\
 &\le c\norm[\Lom2]{\na (w+1)^{\f{m+p-1}2}}^{2b} \norm[{\f{2}{m+p-1}}]{(w+1)^{\f{m+p-1}2}}^{2(1-b)} + c \norm[{\f2{m+p-1}}]{(w+1)^{\f{m+p-1}2}}^2
\end{align*}
holds for every $w\in C^2(\Ombar)$, and, thanks to $b<1$ and $\io (w+1)\le K$, \eqref{estupbynau} follows via an application of Young's inequality. 
% 
% \begin{align*}
%  -\f{n(m+p-1)}{2p} &= (1-\f n2)b - \f{n(m+p-1)}2(1-b)\\
%  \f{n(m+p-1)}2(1-\f1p)&=(1-\f n2+\f{n(m+p-1)}2)b\\
%  
% \end{align*}
% Again, $b>0$ is obvious, at least if $p>2-m$, $p>1$. Also, $p>-\f n2(m-1)=\f n2(1-m)$ implies 
% \begin{align*}
%  -\f n2-\f n2(m-1)\f1p < 1\\
%  \f n2(m+p-1)-\f n2(m+p-1)\f1p<1+\f n2(m+p-1)-\f n2\\
%  \f{n(m+p-1)}2(1-\f1p)<1+\f n2(m+p-2), 
% \end{align*}
% i.e. $b<1$. 
\end{proof}

With the help of this estimate, we have reduced the proof of boundedness by means of Lemma \ref{lem:diffineq} to the following elementary situation. 

\begin{lemma}\label{lem:gronwalllike}
Let $f\colon ℝ\toℝ$ be such that that there exists $x_0\inℝ$ with $f(x)<0$ for any $x>x_0$. Let $y\in C^0([0,T))\cap C^1((0,T))$ for some $T>0$ be such that 
\[
 y'(t)\le f(y(t)) \qquad \text{for any } t\in(0,T).
\]
Then $y(t)\le \max\set{y(0),x_0}$ for any $t\in(0,T)$. % (Note that no Lipschitz requirement is posed on $f$.) 
\end{lemma}
\begin{proof}
Assuming $t\in(0,T)$ to be given, we let $t_0:=\sup\set{s\in[0,t] \mid y(s)\le x_0}$ (or $t_0=0$ in case this set is empty). By definition, we have that $y(s)>x_0$ for all $s\in(t_0,t)$ and $y(t_0)\le \max\{x_0,y(0)\}$. Hence 
\[
 y(t)=y(t_0)+\int_{t_0}^t y'(s)ds \le y(t_0) + \int_{t_0}^t f(y(s)) ds \le y(t_0) + \int_{t_0}^t 0\; ds = y(t_0) \le \max\{x_0,y(0)\}\qedhere 
\]
\end{proof}

% \begin{lemma}
%  If \eqref{eq:sigmasmaller}, then bounded. 
% \end{lemma}
\begin{proof}[Proof of Theorem \ref{thm:bdness}]
Local existence of solutions is guaranteed by Lemma \ref{lem:locex}. If we then combine the differential inequality from Lemma \ref{lem:diffineq} with the estimates of Lemma \ref{lem:upsigma} and Lemma \ref{lem:upgradient}, for any sufficiently large $p$ we obtain $c_1>0$ and $c_2>0$ such that 
\[
 \f{d}{dt} \io (u+1)^p \le c_1 - c_2\kl{\io (u+1)^p}^{\f{m+p-1}p}, 
\]
which, according to Lemma \ref{lem:gronwalllike}, shows boundedness of $\io (u+1)^p$ and hence, by Lemma \ref{lem:pimpliesinfty} boundedness of $u$. 
\end{proof}

\section{The energy functional -- and unboundedness of solutions}

As announced in the introduction, the proof of unboundedness of solutions relies on use of the functional \eqref{defF}, namely on the fact that it decreases along solution trajectories, in the case of global bounded solutions cannot decrease below its lowest value for radially symmetric steady states, but, depending on the initial data, might start from an even lower number. 

We begin by recalling that $\calF$ actually is an energy functional. 

\begin{lemma} \label{lem:Fdecr}
If $(u,v)$ is a classical solution to \eqref{eq:system} with some $D$, $S$ satisfying \eqref{condonDandS}, then the function $\calF$  defined by \eqref{defF} satisfies 
\begin{equation}\label{ddtF}
 \f{d}{dt} \calF(u,v) = - \calD(u,v) \quad \text{on } (0,\Tmax), 
\end{equation}
where 
\begin{equation}\label{defcalD}
 \calD(u,v):=\io S(u) \left|\f{D(u)}{S(u)}\na u - \na v\right|^2.
\end{equation}
\end{lemma}
\begin{proof}
We note that with $G$ as defined in \eqref{defG}
\[
 G'(u)=\int_1^u\f{D(ξ)}{S(ξ)} dξ \qquad \text{for any } u\in(0,∞)
\]
and hence 
\[
 \nabla G'(u)=\f{D(u)}{S(u)}\nabla u \qquad \text{for any positive differentiable function }u.
\]
With \eqref{eq:system} and integration by parts, the calculations are straightforward and we give them without further comment: 
\begin{align*}
 \f{d}{dt} \calF(u,v) &= \io \na v \na v_t + \io vv_t - \io uv_t -\io u_tv + \io G'(u)u_t\\
 &= - \io Δvv_t +\io vv_t - \io uv_t + \io \kl{G'(u)-v} \na\cdot\kl{D(u)\na u - S(u)\na v}\\
 &= -\io (Δv-v+u)v_t - \io \kl{\na G'(u)-\na v}\kl{D(u)\na u-S(u)\na v}\\
 &= - \io \f{D(u)}{S(u)} D(u) |\na u|^2 +\io \f{D(u)}{S(u)}S(u)\na u\cdot\na v + \io D(u)\na u\cdot\na v - \io S(u)|\na v|^2\\
 &= - \io \f{(D(u))^2}{S(u)} |\na u|^2 + 2\io D(u)\na u\cdot \na v -\io S(u)|\na v|^2\\
 &= -\io S(u)\kl{|\na v|^2 -2\f{D(u)}{S(u)} \na u\cdot \na v + \f{(D(u))^2}{(S(u))^2} |\na u|^2}\\
 &= -\io S(u) \left\lvert \f{D(u)}{S(u)}\na u- \na v\right\rvert^2\qquad \text{on } (0,\Tmax).\qedhere
\end{align*}
\end{proof}

We can (and will) simplify the expression for $\calF$ in the particular situation that $u$ and $v$ fulfil \eqref{veq}: 

\begin{lemma}\label{lem:Fsimplified}
 If $(u,v)\in C(\Ombar)\times C^1(\Ombar)$ is such that
\begin{equation}\label{ellipteq}
 0=Δv-v+u,
\end{equation}
 is satisfied in the weak sense, then 
\begin{equation}\label{F:simplified}
 \calF(u,v)=\io G(u) - \f12\io uv.
\end{equation}
\end{lemma}
\begin{proof}
If \eqref{ellipteq} holds, then, upon using $\frac12 v$ as test function, we have 
% \[
%  0=-\io |\na v|^2-\io v^2 + \io uv, 
% \]
% that is 
\[
 \f12\io |\na v|^2 = -\f12 \io v^2 + \f12\io uv.  
\]
If we insert this into the definition of $\calF$, we obtain \eqref{F:simplified}.
\end{proof} 

If a solution $(u,v)$ is global and bounded, $\calF(u(\cdot,t),v(\cdot,t))$ converges, at least along a sequence $t_k\nearrow\infty$.
% If $(u,v)$ is a global bounded solution, then $\omega(u)\cap Equilibria \neq \emptyset$. Although dealing with a slightly different model, the proof of \cite[Lemma 1.2]{win_volumefilling} covers this with a marginal adjustment only. 

\begin{lemma}\label{lem:largetime}
 Assume that $u_0\in C^{β}(\Ombar)$, $β\in(0,1)$ is nonnegative, $S$ and $D$ satisfy \eqref{condonDandS} and $S(0)=0$ and that $(u,v)$ is a global classical solution of \eqref{eq:system} which is bounded in the sense that 
\[
 \sup_{t\in(0,\infty)} \norm[\Lom \infty]{u(\cdot,t)} < \infty. 
\]
Then there are a sequence $(t_k)_{k\inℕ}\nearrow \infty$ and $(u_\infty,v_\infty)\in (C^2(\Ombar))^2$ such that $(u(\cdot,t_k),v(\cdot,t_k))\to (u_\infty,v_\infty)$ in $(C^2(\Ombar))^2$ as $k\to \infty$ and $(u_\infty,v_\infty)$ satisfies 
\begin{equation}\label{largetimeprops}
 \io u_\infty=\io u_0,\quad -\Delta v_\infty+v_\infty=u_\infty,\quad \delny v_\infty\bdry=0,\quad D(u_\infty)\na u_\infty = S(u_\infty)\na v_\infty.
\end{equation}
If $u_0$ is radially symmetric, then also $(u_\infty,v_\infty)$ is radially symmetric.
\end{lemma}
\begin{proof}
 The proof closely follows that of \cite[Lemma 2.2]{win_volumefilling}:
 Boundedness of $u$ makes application of regularity theory possible, yielding $c_1>0$ such that 
\begin{equation}\label{bdnessinC2}
 \normm{C^{2+β,1+\f{β}2}(\Ombar\times[t,t+1])}{u} \le c_1 \quad \text{and}\quad \norm[C^{2+β}(\Ombar)]{v(\cdot,t)}\le c_1 
\end{equation}
for every $t>0$. Due to Arzel\`a--Ascoli's theorem and $\int_0^\infty \calD(u(\cdot,t),v(\cdot,t))dt<\infty$, which is a result of an integration of \eqref{ddtF} and \eqref{bdnessinC2}, we can extract a sequence $(t_k)_{k\inℕ}\nearrow \infty$ such that $\calD(u(\cdot,t_k),v(\cdot,t_k))\to 0$ and $u(\cdot,t_k)\to u_\infty$, $v(\cdot,t_k)\to v_\infty$ in $C^2(\Ombar)$ as $k\to\infty$. 
Apart from \begin{equation}\label{limiteq} D(u_\infty)\na u_\infty = S(u_\infty)\na v_\infty,\end{equation} the properties asserted in \eqref{largetimeprops} immediately follow from the convergence in $C^2(\Ombar)$. In order to show that \eqref{limiteq} holds, we fix $x\in \Om$. If $u_\infty(x)=0$ then also $\nabla u_\infty(x)=0$ due to the nonnegativity of $u_\infty$, so that $S(0)=0$ ensures that \eqref{limiteq} holds in $x$. 
%By taking a subsequence, we may furthermore assume
We have chosen the subsequence such that $\calD(u(\cdot,t_k),v(\cdot,t_k))\to 0$. 
Hence for almost every $x\in \Om$ with $u_\infty(x)\neq0$ by \eqref{condonDandS} we have $\liminf_{k\to\infty} S(u(x,t_k))>0$ and $S(u(x,t_k)) |\f{D(u(x,t_k))}{S(u(x,t_k))}\na u(x,t_k) - \na v(x,t_k)|^2 \to 0$, which shows that $\f{D(u_\infty(x))}{S(u_\infty(x))}\na u_\infty(x) - \na v_\infty(x)=0$ and thus asserts that \eqref{limiteq} holds almost everywhere in $\Om$ and -- by virtue of $u_\infty, v_\infty\in C^2(\Ombar)$ -- in all of $\Om$. 
\end{proof}

% 
% \subsection{stst}
% For radial steady state solutions, $\calF\geq -C$. \cite[Lemma 2.3]{win_volumefilling} ($n=2$), \cite[Lemma 2.4]{win_volumefilling} ($n\ge 3$).

On the other hand, it is impossible to achieve arbitrarily low values of $\calF$ during convergence as observed in Lemma \ref{lem:largetime}.

\begin{lemma}\label{lem:steadyenergygeq}
 Let $N\ge 3$. Then for any $M>0$, $K>0$, $s_0\ge 1$, $δ>0$ and $R>0$ there is $C>0$ such that whenever $S$, $D$ satisfy \eqref{condonDandS} as well as \eqref{eq:winvolumefilling2.11}, 
then every radially symmetric solution to 
\[
 \io u_\infty=M,\quad -\Delta v_\infty+v_\infty=u_\infty,\quad \delny v_\infty\bdry=0,\quad D(u_\infty)\na u_\infty = S(u_\infty)\na v_\infty
\]
in $B_R\sub ℝ^N$ satisfies 
\[
 \calF(u_\infty,v_\infty)>-C. 
\]
\end{lemma}
\begin{proof}
 This is \cite[Lemma 3.4]{win_volumefilling}. Due to its length we refrain from repeating the proof. 
\end{proof}

In combination, the previous lemmata mean that 

\begin{lemma}\label{lem:unbd}
Let $\Om=B_R$ for some $R>0$. Let $D$ and $S$ satisfy \eqref{condonDandS} and $S(0)=0$. Assume that furthermore \eqref{eq:winvolumefilling2.11} is satisfied with some $s_0\ge 1$, $K>0$, $δ>0$. Then there is $C>0$ with the following property: If $u_0\in C^{β}(\Ombar)$ is radially symmetric and such that 
\[
 \calF(u_0,v_0)<-C
\]
 holds for the function $v_0\in C^2(\Ombar)$ defined by 
\begin{equation}\label{diffeqforinitdata}
 0=\Delta v_0 - v_0 + u_0, \qquad \delny v_0\bdry=0, 
\end{equation}
then the corresponding solution is not globally bounded, i.e. blows up, either after finite or in infinite time. 
\end{lemma}
\begin{proof}
Part of Lemma \ref{lem:locex} ensures that the map $\varphi\colon t\mapsto \calF(u(\cdot,t),v(\cdot,t))$  belongs to $C^0([0,\Tmax))$. That \eqref{veq} is satisfied, together with the regularity of $(u,v)$ asserted in Lemma \ref{lem:locex}, serves to show \eqref{diffeqforinitdata} with $v_0:=(C^1(\Ombar)-\lim)_{t\searrow 0} v(\cdot,t)$, firstly in a weak sense, then, by elliptic theory, even classically. According to Lemma \ref{lem:Fdecr}, $\varphi$ is decreasing. Assuming global boundedness of $(u,v)$, the use of Lemma \ref{lem:largetime} leads to $\calF(u_0,v_0)\ge -C$ by 
\ref{lem:steadyenergygeq} (with $C$ as given there). 
\end{proof}

\section{Constructing initial data and estimating $\calF$}\label{sec:construct}

Now that we have established that initial data ``with sufficiently negative energy'' lead to unbounded solutions, what remains to be shown is that such initial data, in fact, do exist and, even more, that there are many of these in any neighbourhood of given initial data. The difficulty, if compared to previous studies of the parabolic-parabolic model, is that $v_0$ can no longer be chosen arbitrarily, but has to fit with $u_0$; this can already be seen from the statement of Lemma \ref{lem:unbd}. 

The goal of this section is to construct one family of functions that causes arbitrarily negative values of $\calF$ if a parameter tends to zero. 
We will later add these functions to given initial data in order to find many nearby initial data that yield blow-up solutions. 

All functions in this section will be radially symmetric; as usual, we will identify radial functions $u\colon \Om=B_R\to ℝ$ and $\utilde\colon[0,R)\to ℝ$ if $u(x)=\utilde(|x|)$, $x\in \Om$, and will use the same symbol $u$ to denote both of these functions. 

We fix $γ>0$ and $δ\in(0,1)$, both of which will be subject to further conditions later, see \eqref{choice:gamma}, \eqref{choice:delta}.
For any $η\in(0,1)$ we let $\re:=η^{δ}$, define the nonnegative Lipschitz-continuous function 
\begin{equation}\label{def:ueta}
 \ue(r) := \begin{cases}
  (r^2+η^2)^{-\f{γ}2}-(\re^2+η^2)^{-\f{γ}2},&r\le \re,\\
  0,& r>\re,
 \end{cases}
\end{equation}
and let $\ve$ be the corresponding solution of 

\begin{equation}\label{def:veta}
 -Δ\ve +\ve =\ue\quad\text{in } \Omega, \qquad \delny \ve\bdry=0,  
\end{equation}
that is, 
\begin{equation}\label{eq:v-radial}
 -r^{1-N}(r^{N-1}{\ve}_r)_r = \ue-\ve\qquad \text{in } (0,R)\qquad \text{with } \ve_r(0)=0,\quad \ve_r(R)=0, 
\end{equation}
where $\ve_r(R)=0$ results from the Neumann boundary condition in \eqref{def:veta} and $\ve_r(0)=0$ is a consequence of the radial symmetry of $\ve$, which in turn is implied by radial symmetry of $\ue$ and uniqueness of solutions to \eqref{def:veta}. 

\subsection{Representation of $\ve$}\label{subsect:repr_v}
Let us first derive a representation formula for $\ve$, on which all estimates will be based. 

Integration of \eqref{eq:v-radial} shows that 
\[
 r^{N-1}\ve_r=\int_0^r s^{N-1}\ve(s)ds - \int_0^r s^{N-1}\ue(s)ds
\]
and hence, due to $\ve_r(0)=0$ 
\[
 \ve_r(r)=r^{1-N}\int_0^rs^{N-1}\ve(s)ds - r^{1-N}\int_0^r s^{N-1} \ue(s) ds,
\]
which by another integration is turned into 
\[
 \ve(r)=\ve(R)+\int_r^R s^{1-N} \int_0^s σ^{N-1}\ue(σ)dσds -\int_r^Rs^{1-N}\int_0^s σ^{N-1}\ve(σ) dσds.
\]

Here we can determine $\ve(R)$ from the fact that -- by integration of \eqref{def:veta} -- the $\Lom 1$-norms of $\ue$ and $\ve$ have to coincide. Using that hence 
\begin{align*}
 \f1\omn \norm[\Lom1]{\ue}&=\int_0^R t^{N-1}v_\eta(t)dt \\
&= \f{R^{N}}{N} v_\eta(R)+\int_0^Rt^{N-1}\int_t^R s^{1-N} \int_0^s σ^{N-1}u_\eta(σ)dσdsdt \\
&\qquad - \int_0^Rt^{N-1}\int_t^R s^{1-N} \int_0^s σ^{N-1}v_\eta(σ)dσdsdt,
%NR^{-N}\norm[\Lom1]{\ue} ,
\end{align*}
we obtain the following representation for $\ve$: 
\begin{align}\label{repr:v}
 \ve(r)&=NR^{-N}\omn^{-1}\norm[\Lom1]{\ue} + NR^{-N}\int_0^Rt^{N-1}\int_t^R s^{1-N}\int_0^sσ^{N-1}\ve(σ)dσdsdt\nn 
\\&\qquad -NR^{-N}\int_0^Rt^{N-1}\int_t^Rs^{1-N}\int_0^sσ^{N-1}\ue(σ)dσdsdt + \int_r^R s^{1-N}\int_0^s σ^{N-1}\ue(σ)dσds \nn 
\\&\qquad - \int_r^Rs^{1-N}\int_0^s σ^{N-1}\ve(σ)dσds.
\end{align}

\subsection{Estimates of $\ve$ from above}\label{subsect:estabove}
Our aim is $\calF(\ue,\ve)\to -\infty$ as $η\to 0$. According to Lemma \ref{lem:Fsimplified}, $\calF(\ue,\ve)=\io G(\ue)-\f12\io \ue\ve$, so that we should prove largeness of $\io \ue\ve$. Estimates of $\ve$ from below would be beneficial to this purpose. Due to the last term in \eqref{repr:v}, which contains $-\ve$, we begin this search for such estimates with an attempt to estimate $\ve$ from above. 

Regardless of whether an estimate from above or below is desired, the first three terms on the right of \eqref{repr:v} have a negligible contribution to the size of $\ve(r)$ if $η$ is small, at least provided $γ<N$:
\begin{align}\label{uvl1est}
 \norm[\Lom1]{\ve}=\norm[\Lom1]{\ue}&= \omn\int_0^R r^{N-1}\ue(r)dr \nn\\
 &\le\omn\int_0^{\re}r^{N-1}(r^2+η^2)^{-\gh}dr\nn\\% - \omn\int_0^{\re}r^{N-1}(\re^2+η^2)^{-\gh}dr\\
&\le \omn\int_0^{\re} r^{N-1-γ}dr\nn\\% - 2^{-\gh}\omn\re^{-γ}\int_0^{\re}r^{N-1}dr\\
&= \f{\omn}{N-γ}\re^{N-γ}  %- \frac{2^{-\gh}\omn}{n}\re^{N-γ} = cη^{δ(N-γ)}\le c,
= \f{\omn}{N-γ}η^{δ(N-γ)},% \le \f{\omn}{N-γ},
\end{align}
%where $c:=\omn\kl{\f1{N-γ}-\f1{2^{\gh}n}}>0$ and 
where we have used the obvious estimate 
\[
 (r^2+η^2)^{-\gh} \le r^{-γ}.%\quad \text{and}\quad (\re^2+η^2)^{-\gh}\ge (2\re^2)^{-\gh}, 
\]
% the latter being valid due to $η<1$, $\re=η^{δ}$ for some $δ<1$. 
% \red{negative Terme ignorieren?}

With this, 
\begin{align*}
 \int_0^Rt^{N-1}\int_t^Rs^{1-N}\int_0^sσ^{N-1}\ue(σ)dσdsdt&\le\int_0^Rt^{N-1}\int_t^Rs^{1-N}\int_0^Rσ^{N-1}\ue(σ)dσdsdt\\
 &=\f1{\omn} \norm[\Lom1]{\ue} \int_0^Rt^{N-1}\int_t^R s^{1-N}dsdt\\
 &\le \f1{\omn} \norm[\Lom1]{\ue} \frac1{N-2} \int_0^R t^{N-1}\cdot t^{2-N} dt\\
 &=\f1{2(N-2)\omn}\norm[\Lom1]{\ue} R^2\\
 &\le \f{R^2}{2(N-2)(N-γ)}η^{δ(N-γ)}.
\end{align*}

By the same calculation we also obtain 
\begin{align}\label{tripleintvest}
 \int_0^Rt^{N-1}\int_t^Rs^{1-N}\int_0^sσ^{N-1}\ve(σ)dσdsdt
 &\le \f{R^2}{2(N-2)(N-γ)}η^{δ(N-γ)}.
\end{align}

As to the term containing $\ue$ and two integrals, we consider the cases of small and slightly larger $r$ separately. For the sake of a unified form of the explicit computations, we assume $γ\neq 2$. For $r\le\re$ we then have
\begin{align*}
 \int_r^Rs^{1-N}\int_0^s σ^{N-1}\ue(σ)dσds&\le \int_r^Rs^{1-N}\int_0^{\min\set{s,\re}} σ^{N-1}(σ^2+η^2)^{-\gh} dσds\\
&\le \int_r^{\re}s^{1-N}\int_0^s σ^{N-1-γ}dσds + \int_{\re}^Rs^{1-N}\int_0^{\re} σ^{N-1-γ}dσds\\
&=\f1{N-γ}\int_r^{\re} s^{1-γ} ds +\f1{N-γ} \re^{N-γ} \int_{\re}^R s^{1-N} ds\\
&=\f1{(N-γ)(2-γ)} (\re^{2-γ}-r^{2-γ}) + \f1{(N-γ)(N-2)}\re^{N-γ} (\re^{2-N}-R^{2-N}),
\end{align*}
whereas in the case $r>\re$
\begin{align*}
 \int_r^Rs^{1-N}\int_0^s σ^{N-1}\ue(σ)dσds&\le \int_r^Rs^{1-N}\int_0^{\min\set{s,\re}} σ^{N-1}(σ^2+η^2)^{-\gh} dσds\\
&\le  \int_r^Rs^{1-N}\int_0^{\re} σ^{N-1-γ}dσds\\
&=\f1{N-γ} \re^{N-γ} \int_{r}^R s^{1-N} ds\\
&=\f1{(N-γ)(N-2)}\re^{N-γ} (r^{2-N}-R^{2-N})\\
&\le \f1{(N-γ)(N-2)} \re^{2-γ}.
\end{align*}
Combined, these estimates show that with some $c_1>0$ 
\begin{equation}\label{doubleintuleq}
 \int_r^Rs^{1-N}\int_0^s σ^{N-1}\ue(σ)dσds \le c_1 r^{2-γ}+c_1 \re^{2-γ}.
\end{equation}

We conclude from \eqref{repr:v} and nonnegativity of $\ue$ and $\ve$ that 
\begin{align}
 \ve(r)&\le NR^{-N}\omn^{-1}\norm[\Lom1]{\ue} + NR^{-N}\int_0^Rt^{N-1}\int_t^R s^{1-N}\int_0^sσ^{N-1}\ve(σ)dσdsdt\nn
\\&\qquad + \int_r^R s^{1-N}\int_0^s σ^{N-1}\ue(σ)dσds \nn
\intertext{ and may invoke \eqref{uvl1est}, \eqref{tripleintvest} and \eqref{doubleintuleq} to continue estimating }
 \ve(r)&\le cη^{δ(N-γ)} + cr^{2-γ} + c\re^{2-γ} \label{eq:veleq}
\end{align}
with some $c>0$.

\subsection{Estimates of $\ve$ from below}\label{subsect:estvbelow}
The pointwise upper estimate of $\ve$ that we have just obtained enables us to treat the last integral in \eqref{repr:v}. Namely, as long as $γ\neq 4$, we have 
\begin{align*}
 \int_r^R s^{1-N}& \int_0^s σ^{N-1}\ve(σ)dσds\\
&\le \kl{cη^{δ(N-γ)}+c\re^{2-γ}}\int_r^Rs^{1-N}\int_0^sσ^{N-1}dσds +c\int_r^Rs^{1-N}\int_0^s σ^{N-1}σ^{2-γ} dσds\\
 &=Cη^{δ(N-γ)}+C\re^{2-γ}+\f{c}{N+2-γ}\int_r^R s^{1-N+N+2-γ}ds \\
 &\le Cη^{δ(N-γ)}+C\re^{2-γ}+\f{c}{(N+2-γ)(4-γ)} \kl{R^{4-γ}-r^{4-γ}} \\
 &\le C+Cη^{δ(2-γ)} + C'r^{4-γ}.
\end{align*}
with $c$ as in \eqref{eq:veleq} and $C>0$, $C'>0$ chosen in the obvious way.

The next term to be estimated is 
\(
\int_r^R s^{1-N} \int_0^s \sigma^{N-1} \ue(\sigma)d\sigma ds
\). Apparently, this term is nonnegative, but since it is this term that has to cause the lower estimate of $\ve$ on which we want to rely in having $\io \ue\ve\to \infty$ as $η\to \infty$, mere nonnegativity would be insufficient. 

We treat the terms arising from the two summands in \eqref{def:ueta} separately and restrict the calculation to small values of $r$. 

Using that $(η^2+σ^2)^{-\gh}\ge (2η^2)^{-\gh}$ for any $σ\le η$, for $r\le η$ we obtain 
\begin{align*}
 \int_r^Rs^{1-N}\int_0^{\min\set{s,\re}} σ^{N-1}(σ^2+η^2)^{-\gh}dσds 
&\ge 2^{-\gh} \int_r^{η}s^{1-N}\int_0^s σ^{N-1} η^{-γ} dσds \\
&\ge \f{2^{-\gh}}{N} η^{-γ} \int_r^{η} sds \\
&\ge c_1η^{2-γ} - c_1r^2η^{-γ}, 
\end{align*}
where $c_1=2^{-\gh-1}N^{-1}$. 

Concerning the second term in \eqref{def:ueta}, for $r<\re$ we have 
\begin{align*}
 \int_r^R s^{1-N}\int_0^{\min\set{\re,s}} σ^{N-1}(\re^2+η^2)^{-\gh} dσds &\le \re^{-γ}\int_r^R s^{1-N} \int_0^{\min\set{s,\re}} σ^{N-1}dσds\\
 &=\re^{-γ}\int_r^{\re} s^{1-N}\int_0^s σ^{N-1} dσds + \re^{-γ}\int_{\re}^R s^{1-N} \int_0^{\re} σ^{N-1} dσds\\
 &= \frac{\re^{-γ}}{N}\int_r^{\re} sds + \frac{\re^{N-γ}}{N} \int_{\re}^R s^{1-N} ds \\
 &=\frac{\re^{-γ}}{2N}(\re^2-r^2) + \frac{\re^{N-γ}}{(N-2)N}(\re^{2-N}-R^{2-N})\\
 &\le \frac1{2N} \re^{2-γ} + \frac1{(N-2)N}\re^{2-γ} = \f1{2(N-2)}\re^{2-γ}. 
\end{align*}

Combining these two estimates, we see that for $r<η$  

\begin{align}\label{eq:iintuegeq}
 \int_r^R s^{1-N} \int_0^s σ^{N-1}\ue(σ)dσds 
&\ge \int_r^Rs^{1-N}\int_0^{\min\set{s,\re}} σ^{N-1}(σ^2+η^2)^{-\gh}dσds \nn\\
&- \int_r^R s^{1-N}\int_0^{\min\set{\re,s}} σ^{N-1}(\re^2+η^2)^{-\gh} dσds\nn\\
&\ge c_1η^{2-γ} - c_1r^2η^{-γ} - \f1{2(N-2)}\re^{2-γ}.
% &\ge 2^{-\gh} \int_r^{η}s^{1-N}\int_0^s σ^{N-1} η^{-γ} dσds - ...\\
% &\ge \f{2^{-\gh}}n η^{-γ} \int_r^{η} sds - ...\\
% &\ge cη^{2-γ} - cr^2η^{-γ} - ...
\end{align}

In conclusion, making use of \eqref{eq:veleq} and \eqref{eq:iintuegeq} we obtain a pointwise lower estimate for $\ve(r)$, for any $r<η$: 
\begin{align*}
 \ve(r)&\ge \int_r^R s^{1-N}\int_0^s σ^{N-1}\ue(σ)dσds\\
& -NR^{-N}\int_0^Rt^{N-1}\int_t^Rs^{1-N}\int_0^sσ^{N-1}\ue(σ)dσdsdt  
\\&\qquad - \int_r^Rs^{1-N}\int_0^s σ^{N-1}\ve(σ)dσds\\
&\ge c_1η^{2-γ} - c_1r^2η^{-γ} - \f1{2(N-2)}\re^{2-γ}
- NR^{-N}\left(\f{R^2}{2(N-2)(N-γ)}η^{δ(N-γ)}\right)
-(C+Cη^{δ(2-γ)} + C'r^{4-γ})\\
&\ge c_1η^{2-γ} - c_1r^2η^{-γ} - c_2 \re^{2-γ} - c_3 -c_4 r^{4-γ}
\end{align*}
for suitably chosen constants $c_1>0$, $c_2>0$, $c_3>0$, $c_4>0$.

\subsection{The estimate for $\io \ue\ve$}\label{subsect:iouvge}

We choose $a\in(0,1)$ such that $c_5:=\f{2^{-\gh}}{N} - \f{a^2}{N+2}>0$. Then applying the previously derived estimates we obtain %assuming that \eta<R, as should be the case. 
\begin{align*}
 \frac1{\omn}\io \ue\ve &=\int_0^R r^{N-1} \ue(r)\ve(r)dr\\
&\ge \int_0^{aη} r^{N-1} \ue(r)\ve(r)dr \\
&\ge \int_0^{aη} r^{N-1} (r^2+η^2)^{-\gh} c_1η^{2-γ} dr 
-\int_0^{aη} r^{N-1} (r^2+η^2)^{-\gh} c_1 r^2 η^{-γ} dr \\
&-\int_0^{aη} r^{N-1} (r^2+η^2)^{-\gh} c_2\re^{2-γ} dr
-\int_0^{aη} r^{N-1} (r^2+η^2)^{-\gh} c_3 dr\\
&-\int_0^{aη} r^{N-1} (r^2+η^2)^{-\gh} c_4 r^{4-γ} dr
-\int_0^{aη} r^{N-1} (\re^2+η^2)^{-\gh} c_1 η^{2-γ} dr\\
&\ge  c_1 η^{2-γ} 2^{-\gh} \int_0^{aη} r^{N-1} η^{-γ} dr \\
&-c_1η^{-2γ}\int_0^{aη} r^{N-1}r^2 dr 
-c_2\re^{2-γ} η^{-γ} \int_0^{aη} r^{N-1} dr
-c_3 η^{-γ}\int_0^{aη} r^{N-1} dr\\
&-c_4 η^{-γ}\int_0^{aη} r^{N-1}r^{4-γ} dr
-c_1η^{2-γ} \re^{-γ} \int_0^{aη} r^{N-1} dr \\
&= 2^{-\gh}\frac{c_1}{N} η^{2-2γ} a^{N} η^{N} 
-\f{c_1}{N+2} η^{-2γ}(aη)^{N+2} 
-\f{c_2}{{N}} \re^{2-γ} η^{-γ} (aη)^{N} \\
&-\f{c_3}{N} η^{-γ} (aη)^{N}
-\f{c_4}{{N}+4-γ} η^{-γ} (aη)^{{N}+4-γ} - \f{c_1}{{N}}η^{2-γ}\re^{-γ}(aη)^{N}\\
&\ge c_1a^{N} (\f{2^{-\gh}}{N} - \f{a^2}{{N}+2}) η^{2-2γ+{N}} \\
&-\f{c_2}{N} η^{N-γ+(2-γ)δ}-\f{c_3}{N} η^{N-γ} 
- \f{c_4}{{N}+4-γ} η^{{N}+4-2γ} - \f{c_1}{N} η^{{N}+2-γ-γδ}
%  &\ge \int_0^{η} r^{N-1} (r^2+η^2)^{-\gh} cη^{2-γ} dr\\
%  &-\int_0^{η} r^{N-1} (r^2+η^2)^{-\gh} cr^2η^{-γ} dr\\
%  &-\int_0^{η} r^{N-1} (r^2+η^2)^{-\gh} (\frac1{2n}+\frac1{n(N-2)}) \re^{2-γ} dr\\
%  &-\int_0^{η} r^{N-1} (\re^2+η^2)^{-\gh} cη^{2-γ} dr
\end{align*}

For small values of $η$, the first of these terms dominates the others if $2-2γ+{N}$ is negative and 
\[
 2-2γ+{N} = \min\{2-2γ+{N}, N-γ+(2-γ)δ, N-γ, {N}+4-2γ, {N}+2-γ-γδ\}, 
\]
which is ensured if $γ>2$, since $δ<1$. 

We can therefore summarize the result of subsections \ref{subsect:estabove} -- \ref{subsect:iouvge} as follows: 

There are $η_0>0$ and $c_0>0$ such that for all $η\in(0,η_0)$: 
\begin{equation}\label{uevegeq}
 \io \ue\ve \ge c_0 η^{2-2γ+{N}}
\end{equation}

\section{An upper bound for the positive contribution to $\calF(\ue,\ve)$. Proof of Theorem \ref{thm:unbounded}}

Under the assumption \eqref{Gleq}, 
\begin{align}\label{Gueleq}
 \io G(\ue) &\le C_G|\Om| + C_G ω_{N}\int_0^{\re} r^{N-1}(r^2+η^2)^{-\gh (2+α)} dr\nn\\
& \le c_1 + c_1 η^{-γ(2+α)}\re^{N} = c_1 + c_1 η^{-γ(2+α) + Nδ}.
\end{align}

If we want the term in \eqref{uevegeq} to dominate that of \eqref{Gueleq}, we have to ensure that 
\[
 2-2γ+N < Nδ-γ(2+α). 
\]
The only remaining step then is to not just use $\ue$, but to approximate any given $u_0$ and to adjust arguments where necessary (in particular in \eqref{Gueleq}). We do this in the following Lemma: 

\begin{lemma}\label{lem:Funbounded}
 Assume, $S$ and $D$ are such that \eqref{condonDandS} and $G$ as defined in \eqref{defG} is such that \eqref{Gleq} is satisfied with some $α\inℝ$ and $C_G>0$. 
 If 
\begin{equation}\label{alge2durchn}
 -α>\f2N,
\end{equation}
the following holds:\\
 Let $p\in[1,\f{2N}{N+2})$ if $α\le -\f4{N+2}$ and $p\in[1,-\f{αN}2)$ if $α> -\f4{N+2}$. 
 Given radially symmetric $u_0\in C^{β}(\Ombar)$ for some $β>0$, % \red{less is possible, because we will merely state a property of approximating functions},
 there are radially symmetric functions $\ue, \ve$ such that $0=\Delta \ve-\ve +\ue$, $\delny\ve\bdry=0$ for any $η\in(0,1)$, and 
\[
 \calF(\ue,\ve) \to -\infty \qquad \text{as } η\searrow 0
\]
 and 
\[
 \norm[\Lom p]{\ue-u_0} \to 0 \qquad \text{as } η\searrow 0.
\]
\end{lemma}

\begin{proof}
Since $2<-Nα$ by \eqref{alge2durchn}, we can choose 
\begin{equation}\label{choice:gamma} 
γ\in\left(\f{N+2}2,N\right)\setminus\{2,4\} 
\end{equation}
 such that 
\begin{equation}\label{ichbrauchnochnegleichungsnummer}
 2<-γα.
\end{equation}
We can, moreover, make this choice in such a way that 
\begin{equation} \label{pkleineralsndg} 
\f {N}{γ}>p,
\end{equation}
because $\f Np>\max\{\f{N+2}2,\f{2}{-α}\}$ by the conditions on $p$. 
In light of \eqref{ichbrauchnochnegleichungsnummer}, it is possible to furthermore choose $δ\in(0,1)$ satisfying 
\begin{equation}\label{choice:delta}
 2+(1-δ)N<-γα
\end{equation}
so that, finally,  
\begin{equation}\label{condition:on:exponents}
 2-2γ+N < Nδ-γ(α+2)
\end{equation}
holds.

With $γ$ and $δ$ as chosen here, we now define $\ue$ according to \eqref{def:ueta} and $\ve$ by \eqref{def:veta}. %\red{$-\Delta \ve ...$, begin of ch. 4}
We then pick a small number $q>0$ such that 
\begin{equation}\label{choice:q} 2-2γ+N<(α+2)q\end{equation}
 and define 
\[
 \uehat:=u_0+\ue+η^q. 
\]
(The last summand will only be needed if $α+2< 0$.) 
We let $v_0$ be the corresponding solution to the Neumann problem of $-Δv_0+v_0=u_0$ and define $\vehat=v_0+\ve+η^q$. By linearity of the elliptic equation, $\vehat$ then solves $-Δ\vehat+\vehat=\uehat$ and furthermore obeys $\delny \vehat\bdry=0$. 

%Given (nonnegative, radially symmetric, continuous) initial data $u_0$, we let $v_0$ be the corresponding solution to the Neumann problem of $-Δv_0+v_0=u_0$ and consider $\uehat=u_0+\ue$ and $\vehat=v_0+\ve$. 

We note that 
\begin{align*}
 \norm[\Lom p]{\uehat-u_0}^p&=\norm[\Lom p]{\ue+η^q}^p\le 2^p\omn\int_0^R r^{N-1}\ue^p(r)dr + 2^p|\Om|η^{pq}\\
 &\le2^p\omn\int_0^{\re}r^{N-1}(r^2+η^2)^{-\f{γp}2}dr+ 2^p|\Om|η^{pq}\\% - \omn\int_0^{\re}r^{N-1}(\re^2+η^2)^{-\gh}dr\\
&\le 2^p\omn\int_0^{\re} r^{N-1-γp}dr + 2^p|\Om|η^{pq}\\% - 2^{-\gh}\omn\re^{-γ}\int_0^{\re}r^{N-1}dr\\
&= \f{2^p\omn}{N-γp}\re^{N-γp} + 2^p|\Om|η^{pq}%- \frac{2^{-\gh}\omn}{n}\re^{N-γ} = cη^{δ(N-γ)}\le c,
= \f{2^p\omn}{N-γp}η^{δ(N-γp)} + 2^p|\Om|η^{pq} \to 0 \qquad \text{as } η\to 0, 
\end{align*}
due to $q>0$ and \eqref{pkleineralsndg}.

% 
% 
% 
% 
% 
% Construct $\ue$ according to \eqref{def:ueta}. Let $\uehat:=$ (cases, according to $α+2<>0$!)
% 
% 
% Then $\norm[p]{.}\to 0$, because ...
% We have that for any $p\in [1,\f{n}{γ})$ (which is a nonempty interval due to \eqref{choice:gamma})
% \[
%  \uehat \to u_0 \qquad \text{in } \Lom p
% \]
% as $η\to 0$. 
% 

If $2+α<0$, then $\uehat\ge η^q$ together with \eqref{Gleq} ensures that 
\[
 \io G(\uehat) \le C_G\io \kl{1+ η^{q(2+α)}} \le C_G |\Om| \kl{1 + η^{q(2+α)}}. 
\]
If $2+α\ge 0$, then we use that with some constant $c_1>0$, $u_0(x)+η^q\le c_1$ for all $x\in \Om$ and $η\in(0,1)$ and employ the estimate 

\[
 G(u_0+\ue+η^q)\le C_G+2^{2+α}C_G(u_0+η^q)^{2+α} + 2^{2+α}C_G \ue^{2+α}\le c_2+c_2 \ue^{2+α}
\]
for suitable $c_2>0$, 
yielding 
\begin{align*}
 \io G(u_0+\ue+\eta^q)&\le c_2|\Om| + c_2 \io \ue^{2+α} \le c_2|\Om|+c_2\omn\int_0^{\re}r^{N-1} (r^2+η^2)^{-\gh\cdot(2+α)} dr\\
 &\le c_2|\Om|+ c_2\omn η^{-γ(2+α)} \int_0^{\re} r^{N-1}dr\\
 &\le c_3 + c_3η^{Nδ-γ(2+α)}
\end{align*}
with some $c_3>0$. 
%It would be possible to estimate $(r^2+η^2)^{-\gh\cdot(2+α)}\le r^{-γ(2+α)}$, but then some condition on the exponent has to be used to ensure integrability close to $r=0$. 

Moreover, $\uehat\ge \ue$ and $\vehat\ge \ve$ and hence 
\[
 \io \uehat\vehat \ge \io \ue\ve \ge c_0 η^{2-2γ+N}
\]
by \eqref{uevegeq}. 
Therefore

\begin{align*}
 \calF(\uehat,\vehat)=\io G(u_0+\ue+η^q) - \f12 \io \uehat\vehat \le c_4 η^{q(2+α)}+ c_5 + c_3η^{Nδ-γ(2+α)} - \f {c_0}2 η^{2-2γ+N},
\end{align*}
where $c_4=C_G|\Om|$ and $c_5=\max\set{c_3,c_4}$. Due to \eqref{choice:gamma} and \eqref{choice:q}, the exponent in the last of these terms is negative and according to \eqref{condition:on:exponents}, it is also the smallest exponent. From negativity of its coefficient, we may immediately conclude 
\begin{equation}%\label{eq:infiniteenergy}
 \calF(\uehat,\vehat)\to -\infty \qquad \text{as  } η\to 0.
\end{equation}
\end{proof}

Theorem \ref{thm:unbounded} now becomes a straightforward consequence: 

\begin{proof}[Proof of Theorem \ref{thm:unbounded}]
We combine Lemma \ref{lem:Funbounded} and Lemma \ref{lem:unbd}. 
\end{proof}

\section{Acknowledgement}
The author acknowledges support of the {\em Deutsche Forschungsgemeinschaft} within the project {\em Analysis of chemotactic cross-diffusion in complex frameworks}. 

{\footnotesize
\bibliographystyle{abbrv}
%\bibliography{lit.bib}

}

\end{document}